\documentclass{article} 
\usepackage{nips14submit_e,times}
\usepackage{hyperref}
\usepackage{url}
\usepackage[numbers]{natbib}

\usepackage{amsmath,amssymb,amsthm,natbib,pgf,tikz,enumitem,caption,subcaption}
\usetikzlibrary{arrows}
\usepackage{macros}
\newtheorem{theorem}{Theorem}

\newtheorem{lemma}[theorem]{Lemma}
\newtheorem{corollary}[theorem]{Corollary}
\newtheorem{proposition}[theorem]{Proposition}
\newtheorem{remark}[theorem]{Remark}

\newcommand{\R}{\mathbb{R}}
\newcommand{\inter}{\cap}
\newcommand{\union}{\cup}

\title{On the Convergence Rate of Decomposable Submodular Function Minimization}

\author{
Robert Nishihara,\,\, Stefanie Jegelka,\,\, Michael I.~Jordan \\
Electrical Engineering and Computer Science\\
University of California\\
Berkeley, CA 94720 \\
\texttt{\{rkn,stefje,jordan\}@eecs.berkeley.edu} \\
}

\nipsfinalcopy 

\begin{document}

\maketitle

\begin{abstract}
%
Submodular functions describe a variety of discrete problems in machine learning, signal processing, and computer vision.
However, minimizing submodular functions poses a number of algorithmic challenges.
Recent work introduced an easy-to-use, parallelizable algorithm for minimizing submodular functions that decompose as the sum of ``simple'' submodular functions.
Empirically, this algorithm performs extremely well, but no theoretical analysis was given.
In this paper, we show that the algorithm converges linearly, and we provide upper and lower bounds on the rate of convergence.
Our proof relies on the geometry of submodular polyhedra and draws on results from spectral graph theory.
\end{abstract}

\section{Introduction}

A large body of recent work demonstrates that many discrete problems in machine learning can be phrased as the optimization of a submodular set function \cite{Bach2013Learning}. 
A set function $F\colon 2^V \to \R$ over a ground set $V$ of $N$ elements is \emph{submodular} if the inequality $F(A) + F(B) \geq F(A \union B) + F(A \inter B)$ holds for all subsets $A, B \subseteq V$. 
Problems like clustering \cite{Narasimhan2007Local}, structured sparse variable selection \cite{bach11}, 
MAP inference with higher-order potentials \cite{kohli09}, 
and corpus extraction problems \cite{lin11corp} can be reduced to the problem of submodular function minimization (SFM), that is
\begin{equation}
  \label{eq:setmin}
  \min_{A \subseteq V} F(A). \tag{P1}
\end{equation}
Although SFM is solvable in polynomial time, existing algorithms can be inefficient on large-scale problems. 
For this reason, the development of scalable, parallelizable algorithms 
has been an active area of research \cite{jegelka11fast,Jegelka2013Reflection,kolmogorov12minimizing,stobbe10efficient}. 
Approaches to solving Problem~\eqref{eq:setmin} are either based on combinatorial optimization or on convex optimization via the \emph{Lov\'asz extension}. 

Functions that occur in practice are usually not arbitrary and frequently possess additional exploitable structure.
For example, a number of submodular functions admit specialized algorithms that solve Problem~\eqref{eq:setmin} very quickly.
Examples include cut functions on certain kinds of graphs, concave functions of the cardinality $|A|$, and functions counting joint ancestors in trees.
 We will use the term \emph{simple} to refer to functions $F$ for which we have a fast subroutine for minimizing $F+s$, where $s \in \mathbb R^N$ is any modular function.
 We treat these subroutines as black boxes.
 Many commonly occuring submodular functions (for example, graph cuts, hypergraph cuts, MAP inference with higher-order potentials \cite{Fix13structured,kohli09,Vicente09joint}, co-segmentation \cite{hochbaum09efficient}, certain structured-sparsity inducing functions \cite{Jenatton11prox}, covering functions \cite{stobbe10efficient}, and combinations thereof) can be expressed as a sum
\begin{equation}
  \label{eq:composite}
  F(A) = \sum\nolimits_{r=1}^R F_r(A)
\end{equation}
of simple submodular functions.
Recent work demonstrates that this structure offers important practical benefits \cite{Jegelka2013Reflection,kolmogorov12minimizing,stobbe10efficient}. 
For instance, it admits iterative algorithms that minimize each $F_r$ separately and combine the results in a straightforward manner (for example, dual decomposition). 

In particular, it has been shown that the minimization of decomposable functions can be rephrased as a \emph{best-approximation problem}, the problem of finding the closest points in two convex sets \cite{Jegelka2013Reflection}. 
This formulation brings together SFM and classical projection methods
and yields empirically fast, parallel, and easy-to-implement algorithms. 
In these cases, the performance of projection methods depends heavily on the specific geometry of the problem at hand and is not well understood in general.
Indeed, while \citet{Jegelka2013Reflection} show good empirical results,
the analysis of this alternative approach to SFM was left as an open problem.



\textbf{Contributions.} 
In this work, we study the geometry of the submodular best-approximation problem and ground the prior empirical results in theoretical guarantees. 
We show that SFM via alternating projections, or block coordinate descent, converges at a \emph{linear rate}. We show that this rate holds for the best-approximation problem, relaxations of SFM, and the original discrete problem. 
More importantly, we prove upper and lower bounds on the worst-case rate of convergence. 
Our proof relies on analyzing angles between the polyhedra associated with submodular functions and draws on results from spectral graph theory.
It offers insight into the geometry of submodular polyhedra that may be beneficial beyond the analysis of projection algorithms.


\textbf{Submodular minimization.}
The first polynomial-time algorithm for minimizing arbitrary submodular functions was a consequence of the ellipsoid method \cite{Grotschel1981Ellipsoid}. 
Strongly and weakly polynomial-time combinatorial algorithms followed \cite{Mccormick2006}. 
The current fastest running times are $O(N^5\tau_1 + N^6)$ \cite{orlin09sfm} in general and $O((N^4\tau_1 + N^5)\log F_{\max})$ for integer-valued functions \cite{iwata03scaling}, where $F_{\max} = \max_A|F(A)|$ and $\tau_1$ is the time required to evaluate $F$. 
Some work has addressed decomposable functions \cite{Jegelka2013Reflection,kolmogorov12minimizing,stobbe10efficient}.
The running times in 
\cite{kolmogorov12minimizing} apply to integer-valued functions and  range from $O((N+R)^2\log F_{\max})$ for cuts to $O((N+Q^2R)(N + Q^2R + QR\tau_2)\log F_{\max})$, where $Q \leq N$ is the maximal cardinality of the support of any $F_r$, and $\tau_2$ is the time required to minimize a simple function.
%
\citet{stobbe10efficient} use a convex optimization approach based on Nesterov's smoothing technique. They achieve a (sublinear) convergence rate of $O(1/k)$ for the discrete SFM problem. Their results and our results do not rely on the function being integral.

\textbf{Projection methods.}
Algorithms based on alternating projections between convex sets (and related methods such as the Douglas--Rachford algorithm) have been studied extensively for solving convex feasibility and best-approximation problems 
\cite{Bauschke1994Dykstra,Bauschke1996Projection,Beck2013Convergence,Deutsch2001Best,Deutsch1994Rate,Gubin1967Method,Halperin1962Product,Tseng1997Alternating,VonNeumann1950Functional}. 
See \citet{Deutsch1992Method} for a survey of applications.
In the simple case of subspaces, the convergence of alternating projections has been characterized in terms of the Friedrichs angle $c_F$ between the subspaces \cite{Bauschke1996Projection,Bauschke2014Rate}.
There are often good ways to compute $c_F$ (see 
Lemma~\ref{lem:angles_and_singular_values}), which allow us to obtain concrete linear rates of convergence for subspaces.
The general case of alternating projections between arbitrary convex sets is less well understood. 
\citet{Bauschke1993Convergence} give a general condition for the linear convergence of alternating projections in terms of the value $\kappa_*$ (defined in Section~\ref{sec:condition_linear_convergence}).
However, except in very limited cases, it is unclear how to compute or even bound $\kappa_*$. 
While it is known that $\kappa_* < \infty$ for polyhedra \cite[Corollary~5.26]{Bauschke1996Projection}, the rate may be arbitrarily slow, and the challenge is to bound the linear rate away from one. 
We are able to give a specific {\em uniform} linear rate for the submodular polyhedra that arise in SFM.

Although both $\kappa_*$ and $c_F$ are useful quantities for understanding the convergence of projection methods, they largely have been studied independently of one another. 
In this work, we relate these two quantities for 
polyhedra, thereby obtaining some of the generality of $\kappa_*$ along with the computability of $c_F$. 
To our knowledge, we are the first to relate $\kappa_*$ and $c_F$ outside the case of subspaces. 
We feel that this connection may be useful beyond the context of submodular polyhedra.



\subsection{Background} 
Throughout this paper, we assume that $F$ is a sum of simple submodular functions $F_1, \ldots, F_R$ and that $F(\emptyset) = 0$. Points $s \in \mathbb R^N$ can be identified with (modular) set functions via $s(A) = \sum_{n \in A}s_n$.
 The \emph{base polytope} of $F$ is defined as the set of all modular functions that are dominated by $F$ and that sum to $F(V)$,
\begin{equation*}
  \label{eq:1}
  B(F) = \{s \in \R^N \,|\, s(A) \leq F(A)\; \text{ for all } A \subseteq V \text{ and } s(V) = F(V) \}. 
\end{equation*}
The \emph{Lov\'asz extension} $f \colon \mathbb R^N \to \mathbb R$ of $F$ can be written as the support function of the base polytope, that is~$f(x) = \max_{s \in B(F)} s^\top x$.
Even though $B(F)$ may have exponentially many faces, the extension $f$ can be evaluated in $O(N \log N)$ time \cite{Edmonds70certain}.
The discrete SFM problem~\eqref{eq:setmin} can be relaxed to the non-smooth convex optimization problem
\begin{equation}
  \label{eq:relax}
  \min_{x \in [0,1]^N} f(x) \;\; \equiv \;\; \min_{x \in [0,1]^N} \sum_{r=1}^R f_r(x), \tag{P2}
\end{equation}
where $f_r$ is the Lov\'asz extension of $F_r$. This relaxation is exact -- rounding an optimal continuous solution yields the indicator vector of an optimal discrete solution. The formulation in Problem~\eqref{eq:relax} is amenable to dual decomposition \cite{Komodakis11mrf} and smoothing techniques \cite{stobbe10efficient}, but suffers from the non-smoothness of $f$ \cite{Jegelka2013Reflection}.
Alternatively, we can formulate a proximal version of the problem
\begin{equation}
  \label{eq:prox}
  \min_{x \in \R^N} f(x) + \tfrac{1}{2}\|x\|^2 \;\; \equiv \;\; \min_{x \in \R^N} \sum_{r=1}^R (f_r(x) + \tfrac{1}{2R}\|x\|^2) . \tag{P3}
\end{equation}
By thresholding the optimal solution of Problem~\eqref{eq:prox} at zero, we recover the indicator vector of an optimal discrete solution \cite{fujishige11minnorm}, \cite[Proposition~8.4]{Bach2013Learning}.
\begin{lemma}{\cite{Jegelka2013Reflection}}\label{lem:duals}
  The dual of the right-hand side of Problem~\eqref{eq:prox} is the best-approximation problem
  \begin{equation}
    \label{eq:dual_problem}
    \min\; \|a - b\|^2 \quad a \in \mathcal{A},\;\; b \in \mathcal{B}, \tag{P4}
  \end{equation}
  where $\mathcal{A} = \{ (a_1, \ldots, a_R) \in \mathbb{R}^{NR} \mid \sum_{r=1}^Ra_r = 0\}$ and $\mathcal B = B(F_1) \times \cdots \times B(F_R)$.
\end{lemma}
Lemma~\ref{lem:duals} 
implies that we can minimize a decomposable submodular function by solving Problem~\eqref{eq:dual_problem}, which means finding the closest points between the subspace $\mathcal{A}$ and the product $\mathcal{B}$ of base polytopes. 
Projecting onto $\mathcal A$ is straightforward because $\mathcal A$ is a subspace. 
Projecting onto $\mathcal B$ amounts to projecting onto each $B(F_r)$ separately. The projection $\Pi_{B(F_r)}z$ of a point $z$ onto $B(F_r)$ may be solved by minimizing $F_r - z$ \cite{Jegelka2013Reflection}.
We can compute these projections easily because each $F_r$ is simple. 

Throughout this paper, we use $\mathcal A$ and $\mathcal B$ to refer to the specific polyhedra defined in Lemma~\ref{lem:duals} (which live in $\mathbb R^{NR}$) and $P$ and $Q$ to refer to general polyhedra (sometimes arbitrary convex sets) in $\mathbb R^D$.
Note that the polyhedron $\mathcal B$ depends on the submodular functions $F_1, \ldots, F_R$, but we omit the dependence to simplify our notation. Our bound will be uniform over all submodular functions.

\section{Algorithm and Idea of Analysis} \label{sec:alg_and_idea}

A popular class of algorithms for solving best-approximation problems are projection methods \cite{Bauschke1996Projection}.
The most straightforward approach uses alternating projections (AP) or block coordinate descent.
Start with any point $a_0 \in \mathcal{A}$, and 
inductively generate two sequences via $b_k = \Pi_{\mathcal{B}}a_k$ and $a_{k+1} = \Pi_{\mathcal{A}}b_k$. Given the nature of $\mathcal A$ and $\mathcal B$, this algorithm is easy to implement and use in our setting, and it solves Problem~\eqref{eq:dual_problem} \cite{Jegelka2013Reflection}.
This is the algorithm that we will analyze.

The sequence $(a_k,b_k)$ will eventually converge to an optimal pair $(a_*,b_*)$.
We say that AP converges linearly with rate $\alpha < 1$ if $\|a_k-a_*\| \le C_1\alpha^k$ and $\|b_k-b_*\| \le C_2\alpha^k$ for all $k$ and for some constants $C_1$ and $C_2$. Smaller values of $\alpha$ are better.

\textbf{Analysis: Intuition.}
%
%
We will provide a detailed analysis of the convergence of AP for the polyhedra $\mathcal A$ and $\mathcal B$. 
To motivate our approach, we first provide some intuition with the following much-simplified setup.
Let $U$ and $V$ be one-dimensional subspaces
spanned by the unit vectors $u$ and $v$ respectively. 
In this case, it is known that AP converges linearly with rate $\cos^2\theta$, where $\theta \in [0,\frac{\pi}{2}]$ is the angle such that $\cos\theta = u^{\top}v$. 
The smaller the angle, the slower the rate of convergence. 
For subspaces $U$ and $V$ of higher dimension, the relevant generalization of the ``angle'' between the subspaces is the 
\emph{Friedrichs angle} \cite[Definition~9.4]{Deutsch2001Best}, whose cosine is given by
\begin{equation} \label{eq:friedrichs_angle}
  c_F(U,V) = \sup \left\{ u^\top v \,|\, u \in U \cap (U \cap V)^{\perp}, v \in V \cap (U \cap V)^{\perp}, \|u\| \le 1, \|v\| \le 1 \right\} .
\end{equation}
In finite dimensions, $c_F(U,V) < 1$. 
In general, when $U$ and $V$ are subspaces of arbitrary dimension, AP will converge linearly with rate $c_F(U,V)^2$ \cite[Theorem~9.8]{Deutsch2001Best}. If $U$ and $V$ are \emph{affine spaces}, AP still converges linearly with rate $c_F(U-u,V-v)^2$, where $u \in U$ and $v \in V$.

We are interested in rates for \emph{polyhedra} $P$ and $Q$, which we define as the intersection of finitely many halfspaces. 
We generalize the preceding results by considering all pairs $(P_x, Q_y)$ of faces of $P$ and $Q$ and showing that the convergence rate of AP between $P$ and $Q$ is at worst $\max_{x,y} c_F(\affnot(P_x), \affnot(Q_y))^2$, where $\aff(C)$ is the affine hull of $C$ and $\affnot(C) = \aff(C)-c$ for some $c \in C$. The faces $\{P_x\}_{x \in \mathbb R^D}$ of $P$ are defined as the nonempty maximizers of linear functions over $P$, that is
\begin{equation} \label{eq:definition_of_face}
  P_x = \argmax_{p \in P} x^\top p .
\end{equation}
While we look at angles between pairs of faces, we remark that \citet{Deutsch2006Rate} consider a different generalization of the ``angle'' between arbitrary convex sets.

\textbf{Roadmap of the Analysis.} Our analysis has two main parts.
First, we relate the convergence rate of AP between polyhedra $P$ and $Q$ to the angles between the faces of $P$ and $Q$. 
To do so, we give a general condition under which AP converges linearly (Theorem~\ref{thm:convex_set_kappa_rate}), which we show depends on the angles between the faces of $P$ and $Q$ (Corollary~\ref{cor:max_kappa_and_min_angle}) in the polyhedral case. 
Second, we specialize to the polyhedra $\mathcal{A}$ and $\mathcal{B}$, 
and we equate the angles
 with eigenvalues of certain matrices and use tools from spectral graph theory to bound the relevant eigenvalues in terms of the conductance of a specific graph. 
This yields a worst-case bound of $1-\frac{1}{N^2R^2}$ on the rate, stated in Theorem~\ref{thm:upper_bound}.

In Theorem~\ref{thm:lower_bound}, we show a lower bound of $1 - \frac{2\pi^2}{N^2R}$ on the worst-case convergence rate.

\section{The Upper Bound}

We first derive an upper bound on the rate of convergence of AP between the polyhedra $\mathcal A$ and $\mathcal B$. 
The results in this section are proved in Appendix~\ref{app:upperbound}.

\subsection{A Condition for Linear Convergence} \label{sec:condition_linear_convergence}

We begin with a condition under which AP between two closed convex sets $P$ and $Q$ converges linearly.
This result is similar to that of \citet[Corollary~3.14]{Bauschke1993Convergence}, but the rate we achieve is twice as fast and relies on slightly weaker assumptions.

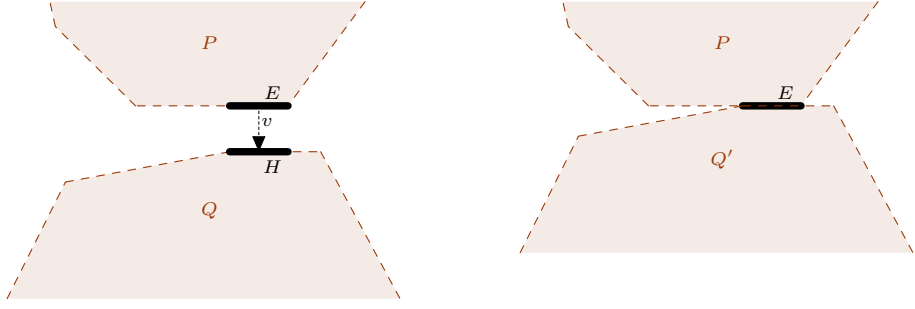
\begin{figure}
\centering
\definecolor{zzttqq}{rgb}{0.6,0.2,0.0}
\begin{subfigure}[t]{0.4\textwidth}
\centering

\vspace{0pt}

\definecolor{zzttqq}{rgb}{0.6,0.2,0.0}
\begin{tikzpicture}[line cap=round,line join=round,>=triangle 45,x=0.5172413793103449cm,y=0.5172413793103449cm]
\fill[color=zzttqq,fill=zzttqq,fill opacity=0.1] (3.42,2.8600000000000003) -- (-0.5000000000000009,2.86) -- (-2.56,4.9) -- (-2.7,5.54) -- (5.3999999999999995,5.54) -- cycle;
\fill[color=zzttqq,fill=zzttqq,fill opacity=0.1] (1.92,1.68) -- (4.26,1.68) -- (6.299999999999998,-2.1) -- (-3.8,-2.1) -- (-2.3,0.9) -- cycle;
\draw [dash pattern=on 3pt off 3pt,color=zzttqq] (3.42,2.8600000000000003)-- (-0.5000000000000009,2.86);
\draw [dash pattern=on 3pt off 3pt,color=zzttqq] (-0.5000000000000009,2.86)-- (-2.56,4.9);
\draw [dash pattern=on 3pt off 3pt,color=zzttqq] (-2.56,4.9)-- (-2.7,5.54);
\draw [dash pattern=on 3pt off 3pt,color=zzttqq] (5.3999999999999995,5.54)-- (3.42,2.8600000000000003);
\draw [dash pattern=on 3pt off 3pt,color=zzttqq] (1.92,1.68)-- (4.26,1.68);
\draw [dash pattern=on 3pt off 3pt,color=zzttqq] (4.26,1.68)-- (6.299999999999998,-2.1);
\draw [dash pattern=on 3pt off 3pt,color=zzttqq] (-3.8,-2.1)-- (-2.3,0.9);
\draw [dash pattern=on 3pt off 3pt,color=zzttqq] (-2.3,0.9)-- (1.92,1.68);
\draw [line width=2.8000000000000003pt] (1.9200000000000002,2.8600000000000003)-- (3.42,2.8600000000000003);
\draw [line width=2.8000000000000003pt] (1.92,1.68)-- (3.42,1.68);
\draw [->,line width=0.3pt,dash pattern=on 1pt off 1pt] (2.67,2.8600000000000003) -- (2.67,1.68);
\begin{scriptsize}
\draw[color=zzttqq] (1.4,4.48) node {$P$};
\draw[color=zzttqq] (1.4,0.18) node {$Q$};
\draw[color=black] (3.02,3.21) node {$E$};
\draw[color=black] (3.02,1.3) node {$H$};
\draw[color=black] (2.8800000000000003,2.4400000000000004) node {$v$};
\end{scriptsize}
\end{tikzpicture}

\end{subfigure}
\hspace{3em}
\begin{subfigure}[t]{0.4\textwidth}
\centering

\vspace{0pt}

\definecolor{zzttqq}{rgb}{0.6,0.2,0.0}
\begin{tikzpicture}[line cap=round,line join=round,>=triangle 45,x=0.5172413793103449cm,y=0.5172413793103449cm]
\fill[color=zzttqq,fill=zzttqq,fill opacity=0.1] (3.42,2.8600000000000003) -- (-0.5000000000000009,2.86) -- (-2.56,4.9) -- (-2.7,5.54) -- (5.3999999999999995,5.54) -- cycle;
\fill[color=zzttqq,fill=zzttqq,fill opacity=0.1] (1.92,2.8600000000000003) -- (4.26,2.8600000000000003) -- (6.299999999999998,-0.9199999999999997) -- (-3.8,-0.9199999999999997) -- (-2.3,2.0800000000000005) -- cycle;
\draw [dash pattern=on 3pt off 3pt,color=zzttqq] (3.42,2.8600000000000003)-- (-0.5000000000000009,2.86);
\draw [dash pattern=on 3pt off 3pt,color=zzttqq] (-0.5000000000000009,2.86)-- (-2.56,4.9);
\draw [dash pattern=on 3pt off 3pt,color=zzttqq] (-2.56,4.9)-- (-2.7,5.54);
\draw [dash pattern=on 3pt off 3pt,color=zzttqq] (5.3999999999999995,5.54)-- (3.42,2.8600000000000003);
\draw [line width=2.8000000000000003pt] (1.9200000000000002,2.8600000000000003)-- (3.42,2.8600000000000003);
\draw [dash pattern=on 3pt off 3pt,color=zzttqq] (1.92,2.8600000000000003)-- (4.26,2.8600000000000003);
\draw [dash pattern=on 3pt off 3pt,color=zzttqq] (4.26,2.8600000000000003)-- (6.299999999999998,-0.9199999999999997);
\draw [dash pattern=on 3pt off 3pt,color=zzttqq] (-3.8,-0.9199999999999997)-- (-2.3,2.0800000000000005);
\draw [dash pattern=on 3pt off 3pt,color=zzttqq] (-2.3,2.0800000000000005)-- (1.92,2.8600000000000003);
\begin{scriptsize}
\draw[color=zzttqq] (1.4,4.48) node {$P$};
\draw[color=black] (3.02,3.21) node {$E$};
\draw[color=zzttqq] (1.4,1.5) node {$Q'$};
\end{scriptsize}
\end{tikzpicture}

\end{subfigure}
\caption{The optimal sets $E$, $H$ in Equation~\eqref{eq:e_h}, the vector $v$, and the shifted polyhedron $Q'$.}
 \label{fig:p_and_q_and_p_and_qprime}
\end{figure}

We will need a few definitions from \citet{Bauschke1993Convergence}. 
Let
$d(K_1,K_2) = \inf \{\|k_1-k_2\| : k_1 \in K_1, k_2 \in K_2\}$ be the distance between sets $K_1$ and $K_2$. Define the sets of ``closest points'' as
\begin{equation}
  \label{eq:e_h}
  E  =  \{p \in P \,|\, d(p,Q) = d(P,Q)\} \quad\qquad
  H  =  \{q \in Q \,|\, d(q,P) = d(Q,P)\},
\end{equation}
and let $v = \Pi_{Q-P} 0$ (see Figure~\ref{fig:p_and_q_and_p_and_qprime}). Note that $H = E+v$, and when $P \cap Q \ne \emptyset$ we have $v = 0$ and $E = H = P \cap Q$. Therefore, we can think of the pair $(E,H)$ as a generalization of the intersection $P \cap Q$ to the setting where $P$ and $Q$ do not intersect. Pairs of points $(e, e+v) \in E \times H$ are solutions to the best-approximation problem between $P$ and $Q$. In our analysis, we will mostly study the translated version $Q' = Q-v$ of $Q$ that intersects $P$ at $E$. 

For $x \in \mathbb R^D \backslash E$, the
function $\kappa$ relates the distance to $E$ with the distances to $P$ and $Q'$, 
\begin{equation*} \label{eq:def_kappa}
  \kappa(x) = \frac{d(x, E)}{\max\{d(x,P), d(x,Q')\}} .
\end{equation*}
If $\kappa$ is bounded, then whenever $x$ is close to both $P$ and $Q'$, it must also be close to their intersection. 
If, for example, $D\ge 2$ and $P$ and $Q$ are balls of radius one whose centers are separated by distance exactly two, then $\kappa$ is unbounded. 
The maximum $\kappa_* = \sup_{x \in (P \cup Q') \backslash E} \kappa(x)$ is useful for bounding the convergence rate. 
\begin{theorem} \label{thm:convex_set_kappa_rate}
  Let $P$ and $Q$ be convex sets, and suppose that $\kappa_* < \infty$. Then AP between $P$ and $Q$ converges linearly with rate $1 - \frac{1}{\kappa_*^2}$. Specifically,
\begin{equation*}
  \label{eq:6}
    \|p_k-p_*\|  \le  2\|p_0-p_*\|( 1 - \tfrac{1}{\kappa_*^{2}} )^k  \quad\text{ and }\quad
  \|q_k-q_*\|  \le  2\|q_0-q_*\|( 1 - \tfrac{1}{\kappa_*^{2}} )^k.
\end{equation*}
\end{theorem}

\subsection{Relating $\kappa_*$ to the Angles Between Faces of the Polyhedra} \label{sec:relating_rate_to_angles}

In this section, we consider the case of polyhedra $P$ and $Q$, and we bound $\kappa_*$ in terms of the angles between pairs of their faces.
In Lemma~\ref{lem:kappa_increasing}, we show that $\kappa$ is nondecreasing along the sequence of points generated by AP between $P$ and $Q'$. We treat points $p$ for which $\kappa(p) = 1$ separately because those are the points from which AP between $P$ and $Q'$ converges in one step. This lemma enables us to bound $\kappa(p)$ by initializing AP at $p$ and bounding $\kappa$ at some later point in the resulting sequence.

\begin{lemma} \label{lem:kappa_increasing}
  For any $p \in P \backslash E$, either $\kappa(p) = 1$ or $1 < \kappa(p) \le \kappa(\Pi_{Q'}p)$. Similarly, for any $q \in Q' \backslash E$, either $\kappa(q) = 1$ or $1 < \kappa(q) \le \kappa(\Pi_P q)$.
\end{lemma}

We can now bound $\kappa$ by angles between faces of $P$ and $Q$.

\begin{proposition} \label{prop:bound_rate_by_angle}
  If $P$ and $Q$ are polyhedra and $p \in P \backslash E$, then there exist faces $P_x$ and $Q_y$ such that
\begin{equation*} \label{eq:bound_rate_by_angle}
  1- \frac{1}{\kappa(p)^2} \le c_F(\affnot(P_x), \affnot(Q_y))^2 .
\end{equation*}
The analogous statement holds when we replace $p \in P \backslash E$ with $q \in Q' \backslash E$.
\end{proposition}
Note that $\affnot(Q_y) = \affnot(Q'_y)$. Proposition~\ref{prop:bound_rate_by_angle} 
immediately gives us the following corollary.
\begin{corollary} \label{cor:max_kappa_and_min_angle}
  If $P$ and $Q$ are polyhedra, then 
\begin{equation*}
  1 - \frac{1}{\kappa_*^2} \le \max_{x,y \in \mathbb R^D} c_F(\affnot(P_x), \affnot(Q_y))^2 .
\end{equation*}
\end{corollary}

\subsection{Angles Between Subspaces and Singular Values} \label{sec:connection_between_angles_and_matrices}
Corollary~\ref{cor:max_kappa_and_min_angle} leaves us with the task of bounding the Friedrichs angle.
To do so, we first relate the Friedrichs angle to the singular values of certain matrices in Lemma~\ref{lem:angles_and_singular_values}. We then specialize this to base polyhedra of submodular functions.
For convenience, we prove Lemma~\ref{lem:angles_and_singular_values} in Appendix~\ref{sec:proof_of_lem_angles_and_singular_values}, though this result is implicit in the characterization of principal angles between subspaces given in \citep[Section~1]{Knyazev2002Principal}.
Ideas connecting angles between subspaces and eigenvalues are also used by \citet{Diaconis2010Stochastic}. 

\begin{lemma} \label{lem:angles_and_singular_values}
  Let $S$ and $T$ be matrices with orthonormal rows and with equal numbers of columns. If all of the singular values of $S T^\top$ equal one, then $c_F(\nullspace(S),\nullspace(T)) = 0$. Otherwise, $c_F(\nullspace(S),\nullspace(T))$ is equal to the largest singular value of $S T^\top$ that is less than one.
\end{lemma}

\textbf{Faces of relevant polyhedra.} 
Let $\mathcal A_x$ and $\mathcal B_y$ be faces of the polyhedra $\mathcal A$ and $\mathcal B$ from Lemma~\ref{lem:duals}.
Since $\mathcal A$ is a vector space, its only nonempty face is $\mathcal A_x = \mathcal A$. Hence, $\mathcal A_x = \nullspace(S)$, where $S$ is an $N \times NR$ matrix of $N \times N$ identity matrices $I_N$:
\begin{equation} \label{eq:def_matrix_S}
  S = \frac{1}{\sqrt{R}}\bigg( \underbrace{\begin{array}{ccc} I_N & \cdots & I_N \end{array}}_{\text{repeated $R$ times}} \bigg) .
\end{equation}
The matrix for $\affnot(\mathcal B_y)$ requires a bit more elaboration. Since $\mathcal B$ is a Cartesian product, we have $\mathcal B_y = B(F_1)_{y_1} \times \cdots \times B(F_R)_{y_R}$, where $y=(y_1, \ldots, y_R)$ and $B(F_r)_{y_r}$ is a face of $B(F_r)$. To proceed, we use the following characterization of faces of base polytopes \cite[Proposition~4.7]{Bach2013Learning}.
\begin{proposition} \label{prop:base_polytope_faces}
  Let $F$ be a submodular function, and let $B(F)_x$ be a face of $B(F)$. Then there exists a partition of $V$ into disjoint sets $A_1, \ldots , A_M$ such that
\begin{equation*}
  \aff(B(F)_x) = \bigcap_{m=1}^M \{s \in \mathbb R^N \,|\, s(A_1 \cup \cdots \cup A_m) = F(A_1 \cup \cdots \cup A_m) \} .
\end{equation*}
\end{proposition}
The following corollary is immediate.
\begin{corollary} \label{cor:base_polytope_faces_translated}
  Define $F$, $B(F)_x$, and $A_1, \ldots, A_M$ as in Proposition~\ref{prop:base_polytope_faces}. Then
\begin{equation*}
  \affnot(B(F)_x) = \bigcap_{m=1}^M \{s \in \mathbb R^N \,|\, s(A_1 \cup \cdots \cup A_m) = 0\} .
\end{equation*}
\end{corollary}
By Corollary~\ref{cor:base_polytope_faces_translated}, for each $F_r$, there exists a partition of $V$ into disjoint sets $A_{r1}, \ldots, A_{rM_r}$ such that
\begin{equation} \label{eq:B_as_intersection_of_vspaces}
  \affnot(\mathcal B_y) = \bigcap_{r=1}^R \bigcap_{m=1}^{M_r} \{(s_1, \ldots, s_R) \in \mathbb R^{NR} \,|\, s_r(A_{r1} \cup \cdots \cup A_{rm}) = 0\} .
\end{equation}
In other words, we can write $\affnot(\mathcal B_y)$ as the nullspace of either of the matrices
\begin{equation*} \label{eq:B_as_null_space_of_matrix}
T' = \left(
  \begin{array}{ccc}
    1_{A_{11}}^\top  & & \\
    \vdots & & \\
    1_{A_{11} \cup \cdots \cup A_{1M_1}}^\top  & & \\
    & \ddots & \\
    & & 1_{A_{R1}}^\top  \\
    & & \vdots \\
    & & 1_{A_{R1} \cup \cdots \cup A_{RM_R}}^\top 
  \end{array}
  \right)
\text{ or }
T = \left(
  \begin{array}{ccc}
    \frac{1_{A_{11}}^\top }{\sqrt{|A_{11}|}}  & & \\
    \vdots & & \\
    \frac{1_{A_{1M_1}}^\top }{\sqrt{|A_{1M_1}|}} & & \\[-5pt]
    & \ddots & \\[-5pt]
    & & \frac{1_{A_{R1}}^\top }{\sqrt{|A_{R1}|}} \\
    & & \vdots \\
    & & \frac{1_{A_{RM_R}}^\top }{\sqrt{|A_{RM_R}|}}
  \end{array}
  \right) ,
\end{equation*}
where $1_A$ is the indicator vector of $A \subseteq V$. For $T'$, this follows directly from Equation~\eqref{eq:B_as_intersection_of_vspaces}. $T$ can be obtained from $T'$ via left multiplication by an invertible matrix,  so $T$ and $T'$ have the same nullspace.
Lemma~\ref{lem:angles_and_singular_values} then implies that $c_F(\affnot(\mathcal A_x),\affnot(\mathcal B_y))$ equals the largest singular value of
\begin{equation*}
  ST^\top  = \frac{1}{\sqrt{R}} \left( \begin{array}{ccccccccc}
    \frac{1_{A_{11}}}{\sqrt{|A_{11}|}}  & \cdots & \frac{1_{A_{1M_1}}}{\sqrt{|A_{1M_1}|}}  & & \cdots & & \frac{1_{A_{R1}}}{\sqrt{|A_{R1}|}}  & \cdots & \frac{1_{A_{RM_R}}}{\sqrt{|A_{RM_R}|}} 
  \end{array}\right) 
\end{equation*}
that is less than one. 
We rephrase this conclusion in the following remark.
\begin{remark} \label{rmk:eigenvalues_and_angles}
  The largest eigenvalue of $(ST^\top )^\top (ST^\top )$ less than one equals $c_F(\affnot(\mathcal A_x),\affnot(\mathcal B_y))^2$.
\end{remark}
Let $M_{\text{all}} = M_1 + \cdots + M_R$. Then $(ST^\top )^\top (ST^\top )$ is the $M_{\text{all}} \times M_{\text{all}}$ square matrix whose rows and columns are indexed by $(r,m)$ with $1 \le r \le R$ and $1 \le m \le M_r$ and whose entry corresponding to row $(r_1,m_1)$ and column $(r_2,m_2)$ equals
\begin{equation*}
  \frac{1}{R}\frac{1_{A_{r_1m_1}}^\top 1_{A_{r_2m_2}}}{\sqrt{|A_{r_1m_1}| |A_{r_2m_2}|}} = \frac{1}{R}\frac{|A_{r_1m_1} \cap A_{r_2m_2}|}{\sqrt{|A_{r_1m_1}| |A_{r_2m_2}|}} .
\end{equation*}

\subsection{Bounding the Relevant Eigenvalues} \label{sec:bounding_eigenvalues}
It remains to bound the largest eigenvalue of $(ST^\top )^\top (ST^\top )$ that is less than one. 
To do so, we view the matrix in terms of the symmetric normalized Laplacian of a weighted graph. 
Let $G$ be the graph whose vertices are indexed by $(r,m)$ with $1 \le r \le R$ and $1 \le m \le M_r$. Let the edge between vertices $(r_1,m_1)$ and $(r_2,m_2)$ have weight $|A_{r_1m_1} \cap A_{r_2m_2}|$. We may assume that $G$ is connected (the analysis in this case subsumes the analysis in the general case). The symmetric normalized Laplacian $\mathcal L$ of this graph is closely related to our matrix of interest,
\begin{equation} \label{eq:outerproduct_as_laplacian}
  (ST^\top )^\top (ST^\top ) = I - \tfrac{R-1}{R} \mathcal L.
\end{equation}
Hence, the largest eigenvalue of $(ST^\top )^\top (ST^\top )$ that is less than one can be determined from the smallest nonzero eigenvalue $\lambda_2(\mathcal{L})$ of $\mathcal{L}$. We bound $\lambda_2(\mathcal{L})$ via Cheeger's inequality (stated in Appendix~\ref{sec:cheeger_inequality_appendix}) by bounding the Cheeger constant $h_G$ of $G$.
\begin{lemma} \label{lem:bound_Cheeger}
  For $R \ge 2$, we have $h_G \ge \frac{2}{NR}$ and hence $\lambda_2(\mathcal{L}) \geq \frac{2}{N^2R^2}$.
\end{lemma}
  We prove Lemma~\ref{lem:bound_Cheeger} in Appendix~\ref{sec:proof_of_lem_bound_cheeger}. 
Combining  Remark~\ref{rmk:eigenvalues_and_angles},
Equation~\eqref{eq:outerproduct_as_laplacian},
and Lemma~\ref{lem:bound_Cheeger}, we obtain the following bound on the Friedrichs angle.
\begin{proposition} \label{prop:bound_on_angle_synthesizing_stuff}
  Assuming that $R \ge 2$, we have
\begin{equation*}
  c_F(\affnot(\mathcal A_x),\affnot(\mathcal B_y))^2 \le 1 - \tfrac{R-1}{R} \tfrac{2}{N^2R^2} \le 1 - \tfrac{1}{N^2R^2} .
\end{equation*}
\end{proposition}
Together with Theorem~\ref{thm:convex_set_kappa_rate} and Corollary~\ref{cor:max_kappa_and_min_angle}, Proposition~\ref{prop:bound_on_angle_synthesizing_stuff} implies the final bound on the rate.
\begin{theorem} \label{thm:upper_bound}
The AP algorithm for Problem~\eqref{eq:dual_problem}
converges linearly with rate $1 - \frac{1}{N^2R^2}$, i.e.,
\begin{equation*}
  \|a_k - a_*\| \leq 2\|a_0 - a_*\| (1 - \tfrac{1}{N^2R^2})^k \;\;\; \text{ and } \;\;\;
  \|b_k - b_*\| \leq 2\|b_0 - b_*\| (1 - \tfrac{1}{N^2R^2})^k.
\end{equation*}
\end{theorem}

\section{A Lower Bound} \label{sec:lower_bound}
To probe the tightness of Theorem~\ref{thm:upper_bound}, we construct a ``bad'' submodular function and decomposition that lead to a slow rate. Appendix~\ref{app:lower_bound} gives the formal details. Our example is an augmented cut function on a cycle:
for each $x,y \in V$, define 
$G_{xy}$ 
to be the cut function of a single edge $(x,y)$,
\begin{equation*}
  \label{eq:7}
  G_{xy} =
  \begin{cases}
    1 & \text{ if } |A \inter \{x,y\}| = 1\\
    0 & \text{ otherwise }.
  \end{cases}
\end{equation*}
Take $N$ to be even and $R \ge 2$ and define the submodular function $F^{\text{lb}} = F_1^{\text{lb}} + \cdots + F_R^{\text{lb}}$, where
\begin{align*}
  \label{eq:worstdecomp}
  F_1^{\text{lb}}  =  G_{12} + G_{34} + \cdots + G_{(N-1)N} \qquad\quad
  F_2^{\text{lb}}  =  G_{23} + G_{45} + \cdots + G_{N1} 
\end{align*}
and $F_r^{\text{lb}} = 0$ for all $r \geq 3$.
The optimal solution to the best-approximation problem is the all zeros vector.

\begin{lemma} \label{lem:lower_bound_example}
The cosine of the Friedrichs angle between $\mathcal A$ and $\aff(\mathcal B^{\text{lb}})$ is
\begin{equation*}
  c_F(\mathcal A, \aff(\mathcal B^{\text{lb}}))^2 = 1 - \tfrac{1}{R}\left( 1 - \cos\left(\tfrac{2\pi}{N}\right)\right).
\end{equation*}
\end{lemma}
Around the optimal solution $0$, the polyhedra $\mathcal{A}$ and $\mathcal{B}^{\text{lb}}$ behave like subspaces, and it is possible to pick initializations $a_0 \in \mathcal{A}$ and $b_0 \in \mathcal{B}^{\text{lb}}$ such that the Friedrichs angle exactly determines the rate of convergence. That means $1 - 1/\kappa_*^2 = c_F(\mathcal A, \aff(\mathcal B^{\text{lb}}))^2$, and 
\begin{equation*}
  \|a_k\| = (1 - \tfrac{1}{R}( 1 - \cos(\tfrac{2\pi}{N})))^k \|a_0\| \quad\text{ and }\quad
  \|b_k\| = (1 - \tfrac{1}{R}( 1 - \cos(\tfrac{2\pi}{N})))^k \|b_0\| .
\end{equation*}
Bounding $1 - \cos(x) \le \frac12 x^2$ leads to the following lower bound on the rate.
\begin{theorem}\label{thm:lower_bound}
  There exists a decomposed function $F^{\text{lb}}$ and initializations for which the convergence rate of AP is at least
  $1 - \frac{2\pi^2}{N^2R}$.
\end{theorem}
This theoretical bound can also be observed empirically (Figure~\ref{fig:lower_bound_plot} in Appendix~\ref{app:lower_bound}).

\section{Convergence of the Primal Objective}
We have shown that AP generates a sequence of points $\{a_k\}_{k \ge 0}$ and $\{b_k\}_{k \ge 0}$ in $\mathbb R^{NR}$ such that $(a_k,b_k) \to (a_*,b_*)$ linearly, where $(a_*,b_*)$ minimizes the objective in Problem~\eqref{eq:dual_problem}. In this section, we show that this result also implies the linear convergence of the objective in Problem~\eqref{eq:prox} and of the original discrete objective in Problem~\eqref{eq:setmin}. The proofs may be found in Appendix~\ref{app:convergence_discrete}.

Define the matrix $\Gamma = -R^{1/2}S$, where $S$ is the matrix defined in Equation~\eqref{eq:def_matrix_S}. Multiplication by $\Gamma$ maps a vector $(w_1, \ldots, w_R)$ to $-\sum_r w_r$, where $w_r \in \mathbb R^N$ for each $r$.
Set $x_k = \Gamma b_k$ and $x_* = \Gamma b_*$. As shown in \citet{Jegelka2013Reflection}, Problem~\eqref{eq:prox} is minimized by $x_*$.
\begin{proposition} \label{prop:primal_convergence}
We have $f(x_k)+\frac12\|x_k\|^2 \to f(x_*)+\frac12\|x_*\|^2$ linearly with rate $1 - \frac{1}{N^2R^2}$.
\end{proposition}

This linear rate of convergence translates into a linear rate for the original discrete problem. 

\begin{theorem} \label{thm:rate_original_discrete_problem}
  Choose $A_* \in \argmin_{A \subseteq V} F(A)$. Let $A_k$ be the suplevel set of $x_k$ 
with smallest value of $F$. Then $F(A_k) \to F(A_*)$ linearly with rate $1 - \frac{1}{2N^2R^2}$.
\end{theorem}

\section{Discussion}
In this work, we analyze projection methods for parallel SFM and give upper and lower bounds on the linear rate of convergence.
 This means that the number of iterations required for an accuracy of $\epsilon$ is logarithmic in $1/\epsilon$, not linear as in previous work \cite{stobbe10efficient}. 
Our rate is uniform over all submodular functions. 
Moreover, our proof highlights how the number $R$ of components and the facial structure of $\mathcal{B}$ affect the convergence rate.
These insights may serve as guidelines when working with projection algorithms and aid in the analysis of special cases. 
For example, reducing $R$ is often possible. Any collection of $F_r$ that have disjoint support, such as the cut functions corresponding to the rows or columns of a grid graph, can be grouped together without making the projection harder. 

Our analysis also shows the effects of additional properties of $F$.
For example, suppose that $F$ is \emph{separable}, that is, $F(V) = F(S) + F(V \backslash S)$ for  some nonempty $S \subsetneq V$. 
Then the subsets $A_{rm} \subseteq V$ defining the relevant faces of $\mathcal B$ satisfy either $A_{rm} \subseteq S$ or $A_{rm} \subseteq S^c$ \cite{Bach2013Learning}. This makes $G$ in Section~\ref{sec:bounding_eigenvalues} disconnected, and as a result,
the $N$ in Theorem~\ref{thm:upper_bound} gets replaced by $\max\{|S|,|S^c|\}$ for an improved rate.
This applies without the user needing to know $S$ when running the algorithm.


A number of future directions suggest themselves. 
For example, \citet{Jegelka2013Reflection} also considered the related Douglas--Rachford (DR) algorithm. 
DR between subspaces converges linearly with rate $c_F$ \cite{Bauschke2014Rate}, as opposed to $c_F^2$ for AP. 
We suspect that our approach may be modified to analyze DR between polyhedra. 
Further questions include the extension to cyclic updates (instead of parallel ones), multiple polyhedra, and stochastic algorithms.

%


\paragraph{Acknowledgments.} We would like to thank M\u{a}d\u{a}lina Persu for suggesting the use of Cheeger's inequality. 
This research is supported in part by NSF CISE Expeditions Award CCF-1139158, LBNL Award 7076018, and DARPA XData Award FA8750-12-2-0331, and  gifts from Amazon Web Services, Google, SAP,  The Thomas and Stacey Siebel Foundation, Apple, C3Energy, Cisco, Cloudera, EMC, Ericsson, Facebook, GameOnTalis, Guavus, HP, Huawei, Intel, Microsoft, NetApp, Pivotal, Splunk, Virdata, VMware, WANdisco, and Yahoo!. 
This work is supported in part by the Office of Naval Research under grant number N00014-11-1-0688, 
the US ARL and the US ARO under grant number W911NF-11-1-0391, 
and the NSF under grant number DGE-1106400.

{\small
\bibliographystyle{abbrvnat}
\bibliography{refs}
}

\appendix

\section{Upper Bound Results}\label{app:upperbound}

\subsection{Proof of Theorem~\ref{thm:convex_set_kappa_rate}} \label{sec:proof_of_convex_kappa_rate}
  For the proof of this theorem, we will need the fact that projection maps are firmly nonexpansive, that is, for a closed convex nonempty subset $C \subseteq \mathbb R^D$, we have
\begin{equation*}
  \|\Pi_Cx - \Pi_Cy\|^2 + \|(x - \Pi_Cx) - (y - \Pi_Cy)\|^2 \le \|x-y\|^2
\end{equation*}
for all $x, y \in \mathbb R^D$. Now, suppose that $\kappa_* < \infty$. Let $e = \Pi_E p_k$ and note that $v = \Pi_Qe - e$ and that $\Pi_Qe \in H$. We have
\begin{align*}
  \kappa_*^{-2} d(p_k, E)^2 & \le d(p_k ,Q')^2 \\
  & \le \|p_k - \Pi_Q p_k + v\|^2 \\
  & \le \|(p_k - \Pi_Q p_k) - (e - \Pi_Q e)\|^2 \\
  & \le \|p_k - e\|^2 - \|\Pi_Qp_k - \Pi_Qe\|^2 \\
  & \le d(p_k, E)^2 - d(q_k, H)^2 .
\end{align*}
It follows that $d(q_k,H) \le (1 - \kappa_*^{-2})^{1/2} d(p_k,E)$. Similarly, we have $d(p_{k+1},E) \le (1 - \kappa_*^{-2})^{1/2} d(q_k,H)$. When combining these, induction shows that
\begin{align*}
  d(p_k, E) & \le ( 1 - \kappa_*^{-2} )^k d(p_0, E) \\
  d(q_k, H) & \le ( 1 - \kappa_*^{-2} )^k d(q_0, H) .
\end{align*}
As shown in \citep[Theorem~3.3]{Bauschke1993Convergence},
the above implies that $p_k \to p_* \in E$ and $q_k \to q_* \in H$ and that
\begin{align*}
  \|p_k-p_*\| & \le 2\|p_0-p_*\|( 1 - \kappa_*^{-2} )^k  \\
  \|q_k-q_*\| & \le 2\|q_0-q_*\|( 1 - \kappa_*^{-2} )^k  .
\end{align*}

\subsection{Connection Between $\kappa$ and $c_F$ in the Subspace Case} \label{sec:connection_between_kappa_cf_in_subspace_case}

In this section, we introduce a simple lemma connecting $\kappa$ and $c_F$ in the case of subspaces $U$ and $V$. We will use this lemma in several subsequent proofs.
\begin{lemma} \label{lem:kappa_bound_in_vspace_case}
  Let $U$ and $V$ be subspaces and suppose $u \in U \cap (U \cap V)^{\perp}$ and that $u \ne 0$. Then
\vspace{-2mm}
\begin{enumerate}[label=(\alph*),noitemsep]
  \item \label{lem:kappa_bound_in_vspace_case:a} $\|\Pi_V u\| \le c_F(U,V)\|u\|$
  \item \label{lem:kappa_bound_in_vspace_case:b} $\kappa(u) \le (1 - c_F(U,V)^2)^{-1/2}$
  \item \label{lem:kappa_bound_in_vspace_case:c} $\kappa(u) = (1 - c_F(U,V)^2)^{-1/2}$ if and only if $\|\Pi_V u\| = c_F(U,V)\|u\|$.
  \end{enumerate}
\vspace{-3mm}
\end{lemma}
\begin{proof}
Part~\ref{lem:kappa_bound_in_vspace_case:a} follows from the definition of $c_F$. Indeed,
\begin{equation*}
  c_F(U,V) \ge \frac{u^\top  (\Pi_V u)}{\|u\|\|\Pi_Vu\|} = \frac{\|\Pi_V u\|^2}{\|u\| \|\Pi_V u\|}  = \frac{\|\Pi_V u\|}{\|u\|} .
\end{equation*}
Part~\ref{lem:kappa_bound_in_vspace_case:b} follows from Part~\ref{lem:kappa_bound_in_vspace_case:a} and the observation that $\kappa(u) = (1 - \|\Pi_V u\|^2/\|u\|^2)^{-1/2}$.
Part~\ref{lem:kappa_bound_in_vspace_case:c} follows from the same observation.
\end{proof}

\subsection{Proof of Lemma~\ref{lem:kappa_increasing}} \label{sec:proof_of_lem_kappa_increasing}

  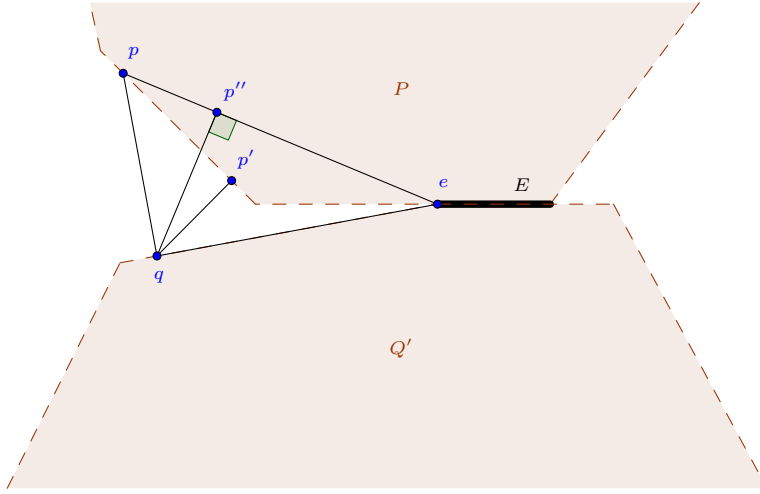
\begin{figure} 
\centering

\definecolor{qqwuqq}{rgb}{0.0,0.39215686274509803,0.0}
\definecolor{qqqqff}{rgb}{0.0,0.0,1.0}
\definecolor{zzttqq}{rgb}{0.6,0.2,0.0}
\begin{tikzpicture}[line cap=round,line join=round,>=triangle 45,x=1.0cm,y=1.0cm]
\fill[color=zzttqq,fill=zzttqq,fill opacity=0.1] (3.42,2.8600000000000003) -- (-0.5000000000000009,2.86) -- (-2.56,4.9) -- (-2.7,5.54) -- (5.3999999999999995,5.54) -- cycle;
\fill[color=zzttqq,fill=zzttqq,fill opacity=0.1] (1.92,2.8600000000000003) -- (4.26,2.8600000000000003) -- (6.299999999999998,-0.9199999999999997) -- (-3.8,-0.9199999999999997) -- (-2.3,2.0800000000000005) -- cycle;
\draw[color=qqwuqq,fill=qqwuqq,fill opacity=0.1] (-1.1214160600333374,3.821131936999301) -- (-0.8603395886883781,3.712324418370924) -- (-0.7515320700600009,3.973400889715883) -- (-1.0126085414049601,4.08220840834426) -- cycle; 
\draw [dash pattern=on 5pt off 5pt,color=zzttqq] (3.42,2.8600000000000003)-- (-0.5000000000000009,2.86);
\draw [dash pattern=on 5pt off 5pt,color=zzttqq] (-0.5000000000000009,2.86)-- (-2.56,4.9);
\draw [dash pattern=on 5pt off 5pt,color=zzttqq] (-2.56,4.9)-- (-2.7,5.54);
\draw [dash pattern=on 5pt off 5pt,color=zzttqq] (5.3999999999999995,5.54)-- (3.42,2.8600000000000003);
\draw [line width=2.8000000000000003pt] (1.9200000000000002,2.8600000000000003)-- (3.42,2.8600000000000003);
\draw [dash pattern=on 5pt off 5pt,color=zzttqq] (1.92,2.8600000000000003)-- (4.26,2.8600000000000003);
\draw [dash pattern=on 5pt off 5pt,color=zzttqq] (4.26,2.8600000000000003)-- (6.299999999999998,-0.9199999999999997);
\draw [dash pattern=on 5pt off 5pt,color=zzttqq] (-3.8,-0.9199999999999997)-- (-2.3,2.0800000000000005);
\draw [dash pattern=on 5pt off 5pt,color=zzttqq] (-2.3,2.0800000000000005)-- (1.92,2.8600000000000003);
\draw (-2.2585437586256125,4.601470518250607)-- (1.92,2.8600000000000003);
\draw (-1.8092555880084846,2.170706313116915)-- (1.92,2.8600000000000003);
\draw (-2.2585437586256125,4.601470518250607)-- (-1.8092555880084846,2.170706313116915);
\draw (-1.8092555880084846,2.170706313116915)-- (-1.0126085414049601,4.08220840834426);
\draw (-1.8092555880084846,2.170706313116915)-- (-0.8163838128200795,3.1733121253169716);
\begin{scriptsize}
\draw[color=black] (3.04,3.12) node {$E$};
\draw [fill=qqqqff] (1.92,2.8600000000000003) circle (1.5pt);
\draw[color=qqqqff] (2.0,3.14) node {$e$};
\draw[color=zzttqq] (1.44,4.4) node {$P$};
\draw[color=zzttqq] (1.44,0.9400000000000001) node {$Q'$};
\draw [fill=qqqqff] (-2.2585437586256125,4.601470518250607) circle (1.5pt);
\draw[color=qqqqff] (-2.12,4.88) node {$p$};
\draw [fill=qqqqff] (-1.8092555880084846,2.170706313116915) circle (1.5pt);
\draw[color=qqqqff] (-1.78,1.9) node {$q$};
\draw [fill=qqqqff] (-1.0126085414049601,4.08220840834426) circle (1.5pt);
\draw[color=qqqqff] (-0.76,4.36) node {$p''$};
\draw [fill=qqqqff] (-0.8163838128200795,3.1733121253169716) circle (1.5pt);
\draw[color=qqqqff] (-0.62,3.4600000000000004) node {$p'$};
\end{scriptsize}
\end{tikzpicture}

\caption{Illustration of the proof of Lemma~\ref{lem:kappa_increasing}.}
\label{fig:p_and_qprime_active}
  \end{figure}

It suffices to prove the statement for $p \in P \backslash E$. For $p \in P \backslash E$, define $q = \Pi_{Q'}p$, $e = \Pi_Eq$, and $p'' = \Pi_{[p,e]}q$, where $[p,e]$ denotes the line segment between $p$ and $e$ (which is contained in $P$ by convexity). 
See Figure~\ref{fig:p_and_qprime_active} for a graphical depiction. 
If $q \in E$, then $\kappa(p) = 1$. So we may assume that $q \notin E$ which also implies that $d(p'',E) > 0$ and $d(\Pi_Pq,E) > 0$. 
We have
\begin{equation} \label{eq:lemma_chain_inequalities}
  \kappa(p) = \frac{d(p,E)}{d(p,Q')} \le \frac{\|p-e\|}{\|p-q\|} \le \frac{\|q-e\|}{\|q-p''\|} \le \frac{d(q,E)}{d(q,P)} = \kappa(q) .
\end{equation}
The first inequality holds because $d(p,E) \le \|p-e\|$ and $d(p,Q') = \|p-q\|$. 
The middle inequality holds because the area of the triangle with vertices $p$, $q$, and $e$ can be expressed as both $\frac12 \|p-e\| \|q-p''\|$ and $\frac12 \|p-q\| \|q-e\| \sin\theta$, where $\theta$ is the angle between vectors $p-q$ and $e-q$, so
\begin{equation*}
  \|p-e\|   \|q-p''\| = \|p-q\| \|q-e\| \sin\theta \le \|p-q\| \|q-e\| .
\end{equation*}
The third inequality holds because $\|q-e\| = d(q,E)$ and $\|q-p''\| \ge d(q,P)$. 
The chain of inequalities in Equation~\eqref{eq:lemma_chain_inequalities} prove the lemma.

\subsection{Proof of Proposition~\ref{prop:bound_rate_by_angle}} \label{sec:proof_of_prop_bound_rate_by_angle}

Suppose that $p \in P \backslash E$ (the case $q \in Q' \backslash E$ is the same), and let $e = \Pi_Ep$. If $\kappa(p) = 1$, the statement is evident, so we may assume that $\kappa(p) > 1$. We will construct sequences of polyhedra
\renewcommand{\tabcolsep}{2pt}
\begin{center}
\begin{tabular}{ccccccc}
  $P$ & $\supseteq$ & $P_1$ & $\supseteq$ & $\cdots$ & $\supseteq$ & $P_J$ \\
  $Q'$ & $\supseteq$ & $Q'_1$ & $\supseteq$ & $\cdots$ & $\supseteq$ & $Q'_J$ 
\end{tabular} .
\end{center}
where $P_{j+1}$ is a face of $P_j$ and $Q'_{j+1}$ is a face of $Q'_j$ for $1 \le j \le J-1$. Either $\dim(\aff(P_{j+1})) < \dim(\aff(P_j))$ or $\dim(\aff(Q'_{j+1})) < \dim(\aff(Q'_j))$ will hold. We will further define $E_j = P_j \cap Q'_j$, which will contain $e$, so that we can define
\begin{equation*}
  \kappa_j(x) = \frac{d(x,E_j)}{\max\{d(x,P_j),d(x,Q'_j)\}}
\end{equation*}
for $x \in \mathbb R^D \backslash E_j$ (this is just the function $\kappa$ defined for the polyhedra $P_j$ and $Q'_j$). Our construction will yield points $p_j \in P_j$, and $q_j \in Q'_j$ such that $p_j \in \relint(P_j) \backslash E_j$, $q_j \in \relint(Q'_j) \backslash E_j$, and $q_j = \Pi_{Q'_j}p_j$ for each $j$. Furthermore, we will have
\begin{equation} \label{eq:kappa_chain}
  \kappa(p) \le \kappa_1(p_1) \le \cdots \le \kappa_J(p_J) .
\end{equation}

Now we describe the construction. For any $t \in [0,1]$, define $p^t = (1-t)p + te$ to be the point obtained by moving $p$ by the appropriate amount toward $e$. Note that $t \mapsto \kappa(p^t)$ is a nondecreasing function on the interval $[0,1)$. Choose $\epsilon > 0$ sufficiently small so that every face of either $P$ or $Q'$ that intersects $B_{\epsilon}(e)$, the ball of radius $\epsilon$ centered on $e$, necessarily contains $e$. Now choose $0 \le t_0 < 1$ sufficiently close to $1$ so that $\|p^{t_0} - e\| < \epsilon$. It follows that $e$ is contained in the face of $P$ whose relative interior contains $p^{t_0}$. It further follows that $e$ is contained in the face of $Q'$ whose relative interior contains $\Pi_{Q'}p^{t_0}$ because
\begin{equation*}
  \|\Pi_{Q'}p^{t_0} - e\| =\|\Pi_{Q'}p^{t_0} - \Pi_{Q'}e\| \le \|p^{t_0}-e\| < \epsilon .
\end{equation*}

To initialize the construction, set
\begin{align*}
  p_1 & = p^{t_0} \\
  q_1 & = \Pi_{Q'}p^{t_0} ,
\end{align*}
and let $P_1$ and $Q'_1$ be the unique faces of $P$ and $Q'$ respectively such that $p_1 \in \relint(P_1)$ and $q_1 \in \relint(Q'_1)$ (the relative interiors of the faces of a polyhedron partition that polyhedron \cite[Theorem~2.2]{Burke1988Identification}). Note that $q_1 \notin E$ because $\kappa(p_1) \ge \kappa(p) > 1$. Note that $e \in E_1 = P_1 \cap Q'_1$ so that
\begin{equation*}
  \kappa(p) \le \kappa(p_1) = \frac{d(p_1,E)}{d(p_1,Q')} = \frac{\|p_1-e\|}{\|p_1-q_1\|} = \frac{d(p_1,E_1)}{d(p_1,Q'_1)} = \kappa_1(p_1) .
\end{equation*}
Now, inductively assume that we have defined $P_j$, $Q'_j$, $p_j$, and $q_j$ satisfying the stated properties. 
Generate the sequences $\{x_k\}_{k \ge 0}$ and $\{y_k\}_{k \ge 0}$ with $x_k \in P_j$ and $y_k \in Q'_j$ by running AP between the polyhedra $P_j$ and $Q'_j$ initialized with $x_0 = p_j$. There are two possibilities, either $x_k \in \relint(P_j)$ and $y_k \in \relint(Q'_j)$ for every $k$, or there is some $k$ for which either $x_k \notin \relint(P_j)$ or $y_k \notin \relint(Q'_j)$. Note that $P_j$ and $Q_j'$ intersect and that AP between them will not terminate after a finite number of steps.

Suppose that $x_k \in \relint(P_j)$ and $y_k \in \relint(Q'_j)$ for every $k$. Then set $J = j$ and terminate the procedure. Otherwise, choose $k'$ such that either $x_{k'} \notin \relint(P_j)$ or $y_{k'} \notin \relint(Q'_j)$. Now set $p_{j+1} = x_{k'}$ and $q_{j+1} = y_{k'}$. Let $P_{j+1}$ and $Q'_{j+1}$ be the unique faces of $P_j$ and $Q'_j$ respectively such that $p_{j+1} \in \relint(P_{j+1})$ and $q_{j+1} \in \relint(Q'_{j+1})$. Note that $p_{j+1},q_{j+1} \notin E_{j+1} = P_{j+1} \cap Q'_{j+1}$ and $e \in E_{j+1}$. We have
\begin{equation*}
  \kappa_j(p_j) < \kappa_j(p_{j+1}) = \frac{d(p_{j+1}, E_j)}{d(p_{j+1}, Q'_j)} = \frac{d(p_{j+1}, E_j)}{\|p_{j+1} - q_{j+1}\|} \le \frac{d(p_{j+1}, E_{j+1})}{d(p_{j+1}, Q'_{j+1})} = \kappa_{j+1}(p_{j+1}) .
\end{equation*}
The preceding work shows the inductive step. Note that if $P_{j+1} \ne P_j$ then $\dim(\aff(P_{j+1})) < \dim(\aff(P_j))$ and if $Q'_{j+1} \ne Q'_j$ then $\dim(\aff(Q'_{j+1})) < \dim(\aff(Q'_j))$. One of these will hold, so the induction will terminate after a finite number of steps.

We have produced the sequence in Equation~\eqref{eq:kappa_chain} and we have created $p_J$, $P_J$, and $Q'_J$ such that AP between $P_J$ and $Q'_J$, when initialized at $p_J$, generates the same sequence of points as AP between $\aff(P_J)$ and $\aff(Q'_J)$. Using this fact, along with \citep[Theorem~9.3]{Deutsch1994Rate}, we see that $\Pi_{\aff(P_J) \cap \aff(Q'_J)}p_J \in E_J$. Using this, along with Lemma~\ref{lem:kappa_bound_in_vspace_case}\ref{lem:kappa_bound_in_vspace_case:b}, we see that
\begin{equation} \label{eq:bound_on_p_J}
  \kappa_J(p_J) \le (1 - c_F(\affnot(P_J),\affnot(Q'_J))^2)^{-1/2} .
\end{equation}
Equations~\eqref{eq:bound_on_p_J} and~\eqref{eq:kappa_chain} prove the result. Note that $P_J$ and $Q'_J$ are faces of $P$ and $Q'$ respectively. We can switch between faces of $Q'$ and faces of $Q$ because doing so amounts to translating by $v$ which does not affect the angles.

\subsection{Proof of Lemma~\ref{lem:angles_and_singular_values}} \label{sec:proof_of_lem_angles_and_singular_values}

We have
\begin{align*}
  c_F(\nullspace(S),\nullspace(T)) & = c_F(\range(S^{\top})^{\perp},\range(T^{\top})^{\perp}) \\
  & = c_F(\range(S^{\top}),\range(T^{\top})) ,
\end{align*}
where the first equality uses the fact that $\nullspace(W) = \range(W^{\top})^{\perp}$ for matrices $W$, and the second equality uses the fact that $c_F(U^{\perp},V^{\perp}) = c_F(U,V)$ for subspaces $U$ and $V$ \cite[Fact~2.3]{Bauschke2014Rate}.

Let $S^{\top}$ and $T^{\top}$ have dimensions $D \times J$ and $D \times K$ respectively, and let $X$ and $Y$ be the subspaces spanned by the columns of $S^{\top}$ and $T^{\top}$ respectively.
Without loss of generality, assume that $J \le K$. Let $\sigma_1 \ge \cdots \ge \sigma_J$ be the singular values of $ST^\top $ with corresponding left singular vectors $u_1, \ldots, u_J$ and right singular vectors $v_1, \ldots, v_J$. Let $x_j = S^{\top}u_j$ and let $y_j = T^{\top}v_j$ for $1 \le j \le J$. By definition, we can write
\begin{equation*}
  \sigma_j = \max_{u,v} \{ u^\top S T^{\top}v \,|\, u \perp \vspan(u_1, \ldots, u_{j-1}), v \perp \vspan(v_1, \ldots, v_{j-1}), \|u\| = 1, \|v\|=1 \} .
\end{equation*}
Since the $\{u_j\}_j$ are orthonormal, so are the $\{x_j\}_j$. Similarly, since the $\{v_j\}_j$ are orthonormal, so are the $\{y_j\}_j$. Suppose that all of the singular values of $S T^{\top}$ equal one. Then we must have $x_j = y_j$ for each $j$, which implies that $X \subseteq Y$, and so $c_F(X,Y) = 0$.

Now suppose that $\sigma_1 = \cdots = \sigma_{\ell} = 1$, and $\sigma_{\ell+1} \ne 1$. It follows that
\begin{equation*}
  X \cap Y = \vspan(x_1, \ldots, x_{\ell}) = \vspan(y_1, \ldots, y_{\ell}) ,
\end{equation*}
and so
\begin{align*}
  \sigma_{\ell+1} & = \sup_{u,v} \{ u^\top S T^{\top}v \,|\, u \in \vspan(u_1, \ldots, u_\ell)^{\perp}, v \in \vspan(v_1, \ldots, v_{\ell})^{\perp}, \|u\| = 1, \|v\| = 1 \} \\
  & = \sup_{x,y} \{ x^\top y \,|\, x \in X \cap (X \cap Y)^{\perp}, y \in Y \cap (X \cap Y)^{\perp}, \|x\| = 1, \|y\| = 1 \} \\
  & = c_F(X,Y) .
\end{align*}

\subsection{Cheeger's Inequality}  \label{sec:cheeger_inequality_appendix}

For an overview of spectral graph theory, see \citet{Chung1997Spectral}. We state Cheeger's inequality below.

Let $G$ be a weighted, connected graph with vertex set $V_G$ and edge weights $(w_{ij})_{i,j \in V_G}$. Define the weighted degree of a vertex $i$ to be $\delta_i = \sum_{j \ne i} w_{ij}$, define the volume of a subset of vertices to be the sum of their weighted degrees, $\vol(U) = \sum_{i \in U} \delta_i$, and define the size of the cut between $U$ and its complement $U^c$ to be the sum of the weights of the edges between $U$ and $U^c$,
\begin{equation*}
  |E(U,U^c)| = \sum_{i \in U, j \in U^c} w_{ij}.
\end{equation*}
The Cheeger constant is defined as
\begin{equation*}
  h_G = \min_{ \emptyset \ne U \subsetneq V_G} \frac{|E(U,U^c)|}{\min(\vol(U),\vol(U^c))} .
\end{equation*}
Let $L$ be the unnormalized Laplacian of $G$, i.e.~the $|V_G| \times |V_G|$ matrix whose entries are defined by
\begin{equation*}
  L_{ij} = \left\{ \begin{array}{ll} -w_{ij} & i \ne j \\ \delta_i & \text{otherwise} \end{array} \right. .
\end{equation*}
Let $D$ be the $|V_G| \times |V_G|$ diagonal matrix defined by $D_{ii} = \delta_i$. Then $\mathcal L = D^{-1/2}LD^{-1/2}$ is the normalized Laplacian. Let $\lambda_2(\mathcal L)$ denote the second smallest eigenvalue of $\mathcal L$ (since $G$ is connected, there will be exactly one eigenvalue equal to zero). 

\begin{theorem}[Cheeger's inequality] \label{thm:cheeger_inequality}
We have $\lambda_2(\mathcal L) \ge \frac{h_G^2}{2}$.
\end{theorem}

\subsection{Proof of Lemma~\ref{lem:bound_Cheeger}} \label{sec:proof_of_lem_bound_cheeger}

\begin{proof}
We have
\begin{align*}
  \min(\vol(U), \vol(U^c)) & \le \frac12 \vol(V_G) \\
  & = \frac12 \sum_{(r,m)} \left(\sum_{(r',m') \ne (r,m)} |A_{rm} \cap A_{r'm'}| \right)\\
  & = \frac12 \sum_{(r,m)} (R-1) |A_{rm}| \\
  & = \frac12 NR(R-1) .
\end{align*}
Since $G$ is connected, for any nonempty set $U \subsetneq V_G$, there must be some element $v \in V$ (here $V$ is the ground set of our submodular function $F$, not the set of vertices $V_G$) such that $v \in A_{r_1m_1} \cap A_{r_2m_2}$ for some $(r_1,m_1) \in U$ and $(r_2,m_2) \in U^c$. Suppose that $v$ appears in $k$ of the subsets of $V$ indexed by elements of $U$ and in $R-k$ of the subsets of $V$ indexed by elements of $U^c$. Then
\begin{equation*}
  |E(U,U^c)| \ge k(R-k) \ge R-1 .
\end{equation*}
It follows that
\begin{equation*}
  h_G \ge \frac{R-1}{\frac12 NR(R-1)} = \frac{2}{NR} .
\end{equation*}
\end{proof}
It follows from Theorem~\ref{thm:cheeger_inequality} that $\lambda_2(\mathcal L) \ge \frac{2}{N^2R^2}$.

\section{Results for the Lower Bound} \label{app:lower_bound}

\subsection{Some Helpful Results} \label{sec:lower_bound_results}

In Lemma~\ref{lem:lower_bound_vector_spaces}, we show how AP between subspaces $U$ and $V$ can be initialized to exactly achieve the worst-case rate of convergence. 
Then in Corollary~\ref{cor:exact_rate_when_same_as_vspace_near_origin}, we show that if subsets $U'$ and $V'$ look like subspaces $U$ and $V$ near the origin,  we can initialize AP between $U'$ and $V'$ to achieve the same worst-case rate of convergence.

\begin{lemma} \label{lem:lower_bound_vector_spaces}
  Let $U$ and $V$ be subspaces with $U \not\subseteq V$ and $V \not\subseteq U$. Then there exists some nonzero point $u_0 \in U \cap (U \cap V)^{\perp}$ such that when we initialize AP at $u_0$, the resulting sequences $\{u_k\}_{k \ge 0}$ and $\{v_k\}_{k \ge 0}$ satisfy
\begin{align*}
  \|u_k\| & = c_F(U,V)^{2k}\|u_0\| \\
  \|v_k\| & = c_F(U,V)^{2k}\|v_0\| .
\end{align*}
\end{lemma}
\begin{proof}
Find $u_* \in U \cap (U \cap V)^{\perp}$ and $v_* \in V \cap (U \cap V)^{\perp}$ with $\|u_*\| = 1$ and $\|v_*\| = 1$ such that $u_*^\top v_* = c_F(U,V)$, which we can do by compactness. By Lemma~\ref{lem:kappa_bound_in_vspace_case}\ref{lem:kappa_bound_in_vspace_case:a},
\begin{equation*}
  c_F(U,V) = v_*^\top u_* = v_*^\top \Pi_Vu_* \le \|\Pi_Vu_*\| \le c_F(U,V).
\end{equation*}
Set $u_0 = u_*$ and generate the sequences $\{u_k\}_{k \ge 0}$ and $\{v_k\}_{k \ge 0}$ via AP. Since $\|\Pi_Vu_0\| = c_F(U,V)$, Lemma~\ref{lem:kappa_bound_in_vspace_case}\ref{lem:kappa_bound_in_vspace_case:c} implies that $\kappa(u_0) = (1-c_F(U,V)^2)^{-1/2}$. Since $\kappa$ attains its maximum at $u_0$, Lemma~\ref{lem:kappa_increasing} implies that $\kappa$ attains the same value at every element of the sequences $\{u_k\}_{k \ge 0}$ and $\{v_k\}_{k \ge 0}$. Therefore, Lemma~\ref{lem:kappa_bound_in_vspace_case}\ref{lem:kappa_bound_in_vspace_case:c} implies that $\|\Pi_Vu_k\| = c_F(U,V) \|u_k\|$ and $\|\Pi_Uv_k\| = c_F(U,V)\|v_k\|$ for all $k$. This proves the lemma.
\end{proof}

\begin{corollary} \label{cor:exact_rate_when_same_as_vspace_near_origin}
  Let $U$ and $V$ be subspaces with $U \not\subseteq V$ and $V \not\subseteq U$. Let $U' \subseteq U$ and $V' \subseteq V$ be subsets such that $U' \cap B_{\epsilon}(0) = U \cap B_{\epsilon}(0)$ and $V' \cap B_{\epsilon}(0) = V \cap B_{\epsilon}(0)$ for some $\epsilon > 0$. Then there is a point $u_0' \in U'$ such that the sequences $\{u_k'\}_{k \ge 0}$ and $\{v_k'\}_{k \ge 0}$ generated by AP between $U'$ and $V'$ initialized at $u_0'$ satisfy
\begin{align*}
  \|u_k'\| & = c_F(U,V)^{2k}\|u_0'\| \\
  \|v_k'\| & = c_F(U,V)^{2k}\|v_0'\| .
\end{align*}
\end{corollary}
\begin{proof}
  Use Lemma~\ref{lem:lower_bound_vector_spaces} to choose some nonzero $u_0 \in U \cap (U \cap V)^{\perp}$ satisfying this property. Then set $u_0' = \frac{\epsilon}{\|u_0\|}u_0$.
\end{proof}

\subsection{Proof of Lemma~\ref{lem:lower_bound_example}} \label{sec:proof_of_lower_bound_example}
Observe that we can write
\begin{equation*}
  \aff(\mathcal B^{\text{lb}}) = \{ (s_1,-s_1, \ldots , s_{\frac{N}{2}},-s_{\frac{N}{2}},-t_{\frac{N}{2}},t_1,-t_1,\ldots,t_{\frac{N}{2}},0,\ldots, 0,\ldots,0,\ldots,0) \,|\, s_i,t_j \in \mathbb R \} .
\end{equation*}
We can write $\aff(\mathcal B^{\text{lb}})$ as the nullspace of the matrix
\begin{equation*}
T_{\text{lb}} = \left(\begin{array}{ccccc} T_{\text{lb},1} & & & & \\ & T_{\text{lb},2} & & & \\ & & I_N & & \\ & & & \ddots & \\ & & & & I_N \end{array} \right),
\end{equation*}
where the $N \times N$ identity matrix $I_N$ is repeated $R-2$ times and where $T_{\text{lb},1}$ and $T_{\text{lb},2}$ are the $\frac{N}{2} \times N$ matrices
\begin{equation*}
  T_{\text{lb},1} = \frac{1}{\sqrt{2}}\left( \begin{array}{ccccccc}
    1 & 1 & & & & & \\
    & & 1 & 1 & & &\\
    & & & & \ddots & & \\
    & & & & & 1 & 1 
  \end{array}\right) \\
  \quad\quad
  T_{\text{lb},2} = \frac{1}{\sqrt{2}}\left( \begin{array}{ccccccc}
    & 1 & 1 & & & & \\
    & & & 1 & 1 & &\\
    & & & & & \ddots & \\
    1 & & & & & & 1 
  \end{array}\right) .
\end{equation*}
Recall that we can write $\mathcal A$ as the nullspace of the matrix $S$ defined in Equation~\eqref{eq:def_matrix_S}. It follows from Lemma~\ref{lem:angles_and_singular_values} that $c_F(\mathcal A,\aff(\mathcal B^{\text{lb}}))$ equals the largest singular value of $ST_{\text{lb}}^\top $ that is less than one. We have
\begin{equation*}
  ST_{\text{lb}}^\top  = \tfrac{1}{\sqrt{R}}\left(\begin{array}{ccccc} T_{\text{lb},1}^\top  & T_{\text{lb},2}^\top  & I_N & \cdots & I_N \end{array} \right) .
\end{equation*}
We can permute the columns of $ST_{\text{lb}}^\top$ without changing the singular values, so $c_F(\mathcal A,\aff(\mathcal B^{\text{lb}}))$ equals the largest singular value of
\begin{equation*}
\tfrac{1}{\sqrt{R}}\left(\begin{array}{cccc} T_{\text{lb},0}^\top  & I_N & \cdots & I_N \end{array} \right), 
\end{equation*}
that is less than one, where $T_{\text{lb},0}$ is the $N \times N$ circulant matrix
\begin{equation*}
  T_{\text{lb},0} = \frac{1}{\sqrt{2}}\left(\begin{array}{cccccc}
    1 & 1 & & & & \\
     & 1 & 1 & & & \\
    & & & \ddots & & \\
    & & & & 1 & 1 \\
    1 & & & & & 1
  \end{array}\right) .
\end{equation*}
Therefore, $c_F(\mathcal A,\aff(\mathcal B^{\text{lb}}))^2$ equals the largest eigenvalue of
\begin{equation*}
  \tfrac{1}{R}\left(\begin{array}{cccc} T_{\text{lb},0}^\top  & I_N & \cdots & I_N \end{array} \right)\left(\begin{array}{cccc} T_{\text{lb},0}^\top  & I_N & \cdots & I_N \end{array} \right)^\top  = \tfrac{1}{R}\left(T_{\text{lb},0}^\top T_{\text{lb},0} + (R-2) I_N \right)
\end{equation*}
that is less than one. Therefore, it suffices to examine the $N \times N$ circulant matrix
\begin{equation*}
  T_{\text{lb},0}^\top T_{\text{lb},0} =
\frac12 \left( \begin{array}{ccccc}
  2 & 1 & & & 1 \\
  1 & 2 & & & \\
  & & \ddots & & \\
  & & & 2 & 1 \\
  1 & & & 1 & 2
\end{array} \right) .
\end{equation*}
The eigenvalues of  $T_{\text{lb},0}^\top T_{\text{lb},0}$ are given by $\lambda_j = 1 + \cos\left(\frac{2\pi j}{N} \right)$ for $0 \le j \le N-1$ (see \citet[Section~3.1]{Gray2006Toeplitz} for a derivation). Therefore,
\begin{equation*}
  c_F(\mathcal A,\aff(\mathcal B^{\text{lb}}))^2 =  1 - \tfrac{1}{R}(1 - \cos(\tfrac{2\pi}{N} ) ) .
\end{equation*}

\subsection{Lower Bound Illustration} \label{sec:lower_bound_plot}

\begin{figure}
  \centering
  \includegraphics[scale=0.9]{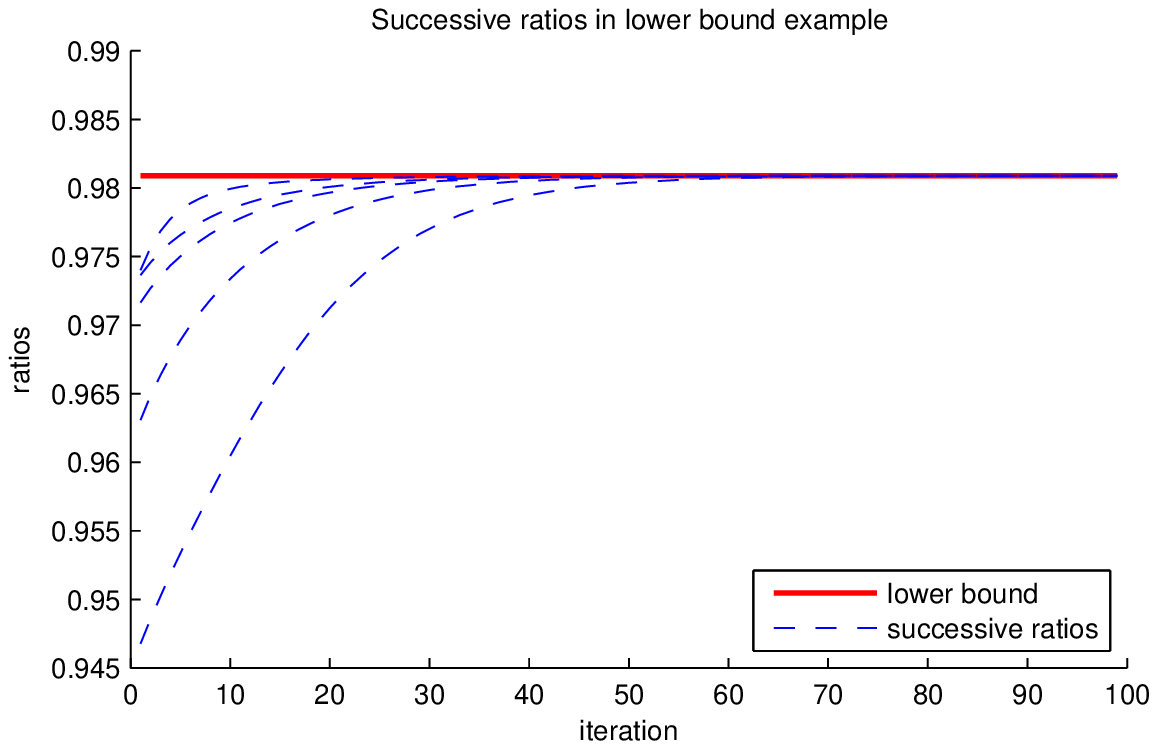}

  \caption{We run five trials of AP between $\mathcal A$ and $\mathcal B^{\text{lb}}$ with random initializations,  where $N = 10$ and $R = 10$. For each trial, we plot the ratios $d(a_{k+1},E)/d(a_k,E)$, where $E =\mathcal A \cap \mathcal B^{\text{lb}}$ is the optimal set. The red line shows the theoretical lower bound of $1 - \tfrac{1}{R}(1-\cos(\tfrac{2\pi}{N}))$ on the worst-case rate of convergence.}
  \label{fig:lower_bound_plot}
\end{figure}

The proof of Theorem~\ref{thm:lower_bound} shows that there is some $a_0 \in \mathcal A$ such that when we initialize AP between $\mathcal A$ and $\mathcal B^{\text{lb}}$ at $a_0$, we generate a sequence $\{a_k\}_{k \ge 0}$ satisfying
\begin{equation*}
  d(a_k,E) = (1 - \tfrac{1}{R}(1-\cos(\tfrac{2\pi}{N}))^kd(a_0,E) ,
\end{equation*}
where $E = \mathcal A \cap \mathcal B^{\text{lb}}$ is the optimal set.
In Figure~\ref{fig:lower_bound_plot}, we plot the theoretical bound in red, and in blue the successive ratios $d(a_{k+1},E)/d(a_k,E)$ for five runs of AP between $\mathcal A$ and $\mathcal B^{\text{lb}}$ with random initializations. 
Had we initialized AP at $a_0$, the successive ratios would exactly equal $1 - \tfrac{1}{R}(1-\cos(\tfrac{2\pi}{N}))$.
The plot of these ratios would coincide with the red line in Figure~\ref{fig:lower_bound_plot}.

Figure~\ref{fig:lower_bound_plot} illustrates that the empirical behavior of AP between $\mathcal A$ and $\mathcal B^{\text{lb}}$ is often similar to the worst-case behavior, even when the initialization is random. 
When we initialize AP randomly, the successive ratios appear to increase to the lower bound and then remain constant. 
Figure~\ref{fig:lower_bound_plot} shows the case $N=10$ and $R=10$, but the plot looks similar for other $N$ and $R$. 

We also note that the graph corresponding to our lower bound example actually achieves a Cheeger constant similar to the one used in Lemma~\ref{lem:bound_Cheeger}.

\section{Results for Convergence of the Primal and Discrete Problems}\label{app:convergence_discrete}

\subsection{Proof of Proposition~\ref{prop:primal_convergence}} \label{sec:proof_of_prop_primal_convergence}
First, suppose that $s \in B(F)$. Let $A = \{n \in V\,|\, s_n \ge 0\}$ be the set of indices on which $s$ is nonnegative. Then we have
\begin{equation} \label{eq:bound_norm_of_base_polytope}
  \|s\| \le \|s\|_1 = 2s(A) - s(V) \le 3F_{\max} .
\end{equation}

Recall that we defined $F_{\max} = \max_A|F(A)|$. 
Now, we show that $f(x_k) + \frac12\|x_k\|^2$ converges to $f(x_*) + \frac12\|x_*\|^2$ linearly with rate $1-\frac{1}{N^2R^2}$.
We will use Equation~\eqref{eq:bound_norm_of_base_polytope} to bound the norms of $x_k$ and $x_*$, both of which lie in $-B(F)$. 
We will also use the fact that $\|x_k-x_*\| \le \|\Gamma\|\|b_k - b_*\| \leq \sqrt{R}\|b_k-b_*\|$.
Finally, we  will use the proof of Theorem~\ref{thm:upper_bound} to bound $\|b_k-b_*\|$.
First, we bound the difference between the squared norms using convexity. We have
\begin{align}
  \nonumber
  \tfrac12\|x_k\|^2 - \tfrac12\|x_*\|^2 & \le x_k^{\top}(x_k-x_*) \\
  \nonumber
  & \le \|x_k\| \|x_k-x_*\| \\
  \nonumber
  & \le 3F_{\max} \sqrt{R} \|b_k-b_*\| \\
  \label{eq:bd1}
  & \le 6F_{\max}\sqrt{R} \|b_0-b_*\| (1 - \tfrac{1}{N^2R^2})^k .
\end{align}
Next, we bound the difference in Lov\'asz extensions. Choose $s \in \argmax_{s \in B(F)} s^\top x_k$. Then
\begin{align}
  \nonumber
  f(x_k)-f(x_*) & \le s^\top (x_k-x_*) \\
  \nonumber
  & \le \|s\| \|x_k-x_*\| \\
  \nonumber
  & \le 3F_{\max} \sqrt{R} \|b_k-b_*\| \\
  \label{eq:bd2}
  & \le 6F_{\max} \sqrt{R} \|b_0-b_*\|(1 - \tfrac{1}{N^2R^2})^k .
\end{align}
Combining the bounds \eqref{eq:bd1} and \eqref{eq:bd2}, we find that
\begin{equation}
  (f(x_k) + \tfrac12\|x_k\|^2) - (f(x_*) + \tfrac12\|x_*\|^2) \le 12F_{\max} \sqrt{R} \|b_0-b_*\|(1 - \tfrac{1}{N^2R^2})^k .
\end{equation}

\subsection{Proof of Theorem~\ref{thm:rate_original_discrete_problem}} \label{sec:proof_of_rate_original_discrete_problem}

We will make use of the following result, shown in \cite[Proposition~10.5]{Bach2013Learning} and stated below for convenience.
\begin{proposition}
  Let $(w,s) \in \mathbb R^N \times B(F)$ be a pair of primal-dual candidates for the minimization of $\frac12\|w\|^2 + f(w)$, with duality gap $\epsilon=\frac12\|w\|^2+f(w)+\frac12\|s\|^2$. Then if $A$ is the suplevel set of $w$ with smallest value of $F$, then
\begin{equation*}
  F(A) - s_-(V) \le \sqrt{N\epsilon/2} .
\end{equation*}
\end{proposition}

Using this result in our setting, recall that by definition $A_k$ is the set of the form $\{n \in V \,|\, (x_k)_n \ge c\}$ for some constant $c$ with smallest value of $F(\{n \in V \,|\, (x_k)_n \ge c\})$.

Let $(w_*,s_*) \in \mathbb R^N \times B(F)$ be a primal-dual optimal pair for the left-hand version of Problem~\eqref{eq:prox}. The dual of this minimization problem is the projection problem $\min_{s \in B(F)} \frac12 \|s\|^2$.
From \cite[Proposition~10.5]{Bach2013Learning}, we see that
\begin{align*}
  F(A_k) - F(A_*) & \le F(A_k) - (s_*)_-(V) \\
  & \le \sqrt{\tfrac{N}{2} ( (f(x_k) + \tfrac12\|x_k\|^2) - (f(x_*) + \tfrac12\|x_*\|^2) ) } \\
  & \le \sqrt{6F_{\max} NR^{1/2} \|b_0-b_*\|}\, (1 - \tfrac{1}{N^2R^2})^{k/2} \\
  & \le \sqrt{6F_{\max} NR^{1/2} \|b_0-b_*\|} (1 - \tfrac{1}{2N^2R^2})^{k},
\end{align*}
where the third inequality uses the proof of Proposition~\ref{prop:primal_convergence}.
The second inequality 
relies on \citet[Proposition~10.5]{Bach2013Learning}, which states that a duality gap of $\epsilon$ for the left-hand version of Problem~\eqref{eq:prox} turns into a duality gap of $\sqrt{N\epsilon/2}$ for the original discrete problem. If our algorithm converged with rate $\frac{1}{k}$, this would translate to a rate of $\frac{1}{\sqrt{k}}$ for the discrete problem. But fortunately, our algorithm converges linearly, and taking a square root preserves linear convergence.

\subsection{Running times}
Theorem~\ref{thm:rate_original_discrete_problem} implies that the number of iterations required for an accuracy of $\epsilon$ is at most
\begin{equation}\label{eq:time}
  2N^2R^2 \log\left( \frac{\sqrt{6F_{\max} N R^{1/2} \|b_0-b_*\|}}{\epsilon} \right) .
\end{equation}
Each iteration involves minimizing each of the $F_r$ separately. 
For comparison, the number of iterations required in \citet{stobbe10efficient} is
\begin{equation*}
  24 \sqrt{NR} \frac{F_{\max}  }{\epsilon} .
\end{equation*}
The dependence of this algorithm on $N$ and $R$ is better, but its dependence on $F_{\max}/\epsilon$ is worse. For example, to obtain the exact discrete solution, we need $\epsilon < \min_{S,T} | F(S) - F(T)|$. This is one for integer-valued functions (in which case the lower rate may be desirable), but can otherwise become very small. The constant $F_{\max}$ can be of order $O(N)$ in general (or even larger if the function becomes very negative). For empirical comparisons, we refer the reader to \cite{Jegelka2013Reflection}.

The running times of the combinatorial algorithm by
\citet{kolmogorov12minimizing} apply to \emph{integer-valued} functions (as opposed to the generic ones above) and  range from $O((N+R)^2\log F_{\max})$ for cuts to $O((N+Q^2R)(N + Q^2R + QR\tau_2)\log F_{\max})$, where $Q \leq N$ is the maximal cardinality of the support of any $F_r$, and $\tau_2$ is the time required to minimize a simple function.
This is better than \eqref{eq:time} if $Q$ is a small constant, and worse as $Q$ gets closer to $N$.

For comparison, if not exploiting decomposition, one may use combinatorial algorithms, the Frank-Wolfe algorithm (conditional gradient descent), or a subgradient method.
The combinatorial algorithm by \citet{orlin09sfm} has a running time of $O(N^5\tau_1 + N^6)$, and the algorithm by \citet{iwata03scaling} (for integer-valued functions) has a running time of $O((N^4\tau_1 + N^5)\log F_{\max})$, where $\tau_1$ is the time required to evaluate $F$. 
For an accuracy of $\epsilon$ in the discrete objective, Frank-Wolfe will take $64N \frac{F_{\max}}{\epsilon^2}$ iterations, each taking time $O(N \log N)$. The subgradient method behaves similarly.
\end{document}